\documentclass[reqno, 12pt]{amsart}
\usepackage{amssymb,amsmath,amsfonts,amsthm,comment,mathrsfs,times,graphicx}
\usepackage[bookmarksnumbered, colorlinks, plainpages]{hyperref}
\usepackage{color}
\usepackage[english]{babel}
\usepackage[all,cmtip]{xy}
\usepackage{lmodern}
\usepackage{enumitem}
\usepackage{geometry}
\newcounter{dummy} \numberwithin{dummy}{section}
\newtheorem{theo}[dummy]{Theorem }
 \newtheorem{coro}[dummy]{Corollary}
 \newtheorem{lem}[dummy]{Lemma}
 \newtheorem{pro}[dummy]{Proposition}
\newtheorem{deft}[dummy]{Definition}

\newtheorem{rem}[dummy]{Remark}

\title[ The index of Rubin-Stark units ]{The index of Rubin-Stark units }
\author[Saad El Boukhari \& Youness Mazigh]{Saad El Boukhari$^{(1)}$ \& Youness Mazigh$^{(1),(2)}$}
\address{$^{(1)}$ Université Moulay Ismaïl,
 Département de mathématiques,
Faculté des sciences de Meknès, B.P. 11201 Zitoune, Meknès, Maroc.}
\email{\textcolor[rgb]{0.00,0.00,1.00}{y.mazigh@edu.umi.ac.ma}}
\address{$^{(2)}$ Université Franche-comté,
Laboratoire de mathématique, 16 Route de Gray, 25030 Besançon,
 cedex, France.}
 \email{\textcolor[rgb]{0.00,0.00,1.00}{youness.mazigh@univ-fcomte.fr}}
\begin{document}
\maketitle
\renewcommand{\abstractname}{Abstract}
\begin{abstract} The aim of this paper is to compare  the orders of the class groups and the quotients of the
$r$-th exterior power of units modulo Rubin-Stark units.
\end{abstract}
\section{Introduction and Preliminares}
The class number associated with a number field is known to be
related to $L$-functions, and this can provide valuable information
about class groups using computations of special values of those
functions. A direct way to link those two concepts is based on what
is called class number formulas. \vskip 6pt Class number formulas
where the class number is compared to the index of special units
within their group of units have been formulated in the abelian and
imaginary cases for circular and elliptic units respectively. It
seems, however, that such results that would use the Rubin-Stark
units are absent from literature and it is in this perspective that
this work has been conducted.\vskip 6pt
 This paper has therefore for aim to
formulate and prove a class number formula which involves the index
of Rubin-Stark units within the group of $S$-units. We introduce
first some of the notations that will be used for this
purpose.\vskip 6pt Let $k$ be a totally real field of degree
$r=[k:\mathbb{Q}]$ and let $K/k$ be a finite abelian extension of
totally real number fields with Galois group $G$. Fix a finite set
$S$ of places of $k$ containing infinite places and all places
ramified in $K/k$, and a second finite set $T$ of places of $k$,
disjoint from $S$. Let
$\widehat{G}=\mathrm{Hom}(G,\mathbb{C}^{\times})$. If
$\chi\in\widehat{G}$, the modified Artin $L$-function attached to
$\chi$ is defined for $s\in \mathbb{C},\; \mathrm{Re}(s)>1$ by
\begin{equation*}
L_{S,T}(s,\chi)=\prod_{\mathfrak{p}\not\in
S}(1-\chi(\sigma_{\mathfrak{p}})\mathbf{N}\mathfrak{p}^{-s})^{-1}
\prod_{\mathfrak{p}\in
T}(1-\chi(\sigma_{\mathfrak{p}})\mathbf{N}\mathfrak{p}^{1-s}),
\end{equation*}
where $\sigma_{\mathfrak{p}}\in G$ is the Frobenius of the
(unramified) prime $\mathfrak{p}$. This function can be analytically
continued to a meromorphic function on $\mathbb{C}$.
\par For each $\chi\in
\widehat{G}$, there is an idempotent
\begin{equation*}
e_{\chi}=\frac{1}{|G|}\sum_{\sigma\in G}\chi(\sigma)\sigma^{-1}\in
\mathbb{C}[G].
\end{equation*}
Following \cite{Tate84} we define the Stickelberger element
\begin{equation*}
\Theta_{S,T}(s)=\Theta_{S,T,K/k}(s)=\sum_{\chi\in\widehat{G}}L_{S,T}(s,\chi^{-1})e_{\chi}
\end{equation*}
which is viewed as a $\mathbb{C}[G]$-valued meromorphic function on
$\mathbb{C}$.
  Let $\chi\in \widehat{G}$ and let
$r_{S}(\chi)$ be  the order of vanishing of $L_{S,T}(s,\chi)$ at
$s=0$. Recall that
\begin{equation*}
r_{S}(\chi)=\mathrm{ord}_{s=0}L_{S,T}(s,\chi)=\left\{
                                              \begin{array}{ll}
                                                |\{v\in S\;:\; \chi(D_{v}(K/k))=1\}|, & \hbox{$\chi\neq 1$;} \\
                                                |S|-1, & \hbox{$\chi=1$.}
                                              \end{array}
                                            \right.
\end{equation*}
(see \textit{e.g}.\, \cite[Proposition \textrm{I}.3.4]{Tate84}),
where $D_{v}(K/k)$ is the decomposition group of $v$ relative to
$K/k$.\vskip 6pt Before stating the Rubin-Stark conjecture we record
some hypotheses $\mathbf{H}(K/k,S,T,r)$:
\begin{enumerate}
    \item $S$ contains all the infinite primes of $k$ and all the
    primes which ramify in $K/k$;
    \item $S$ contains at least $r$ places which split completely in
    $K/k$;
    \item $|S|\geq r+1$;
    \item $T\neq \emptyset$, $S\cap T=\emptyset$ and $U_{S,T}(K)$ is
    torsion-free.
\end{enumerate}
Here $U_{S,T}(K)$ is the group of $S$-units of $K$ which are
congruent to $1$ modulo all the primes in $T$.\vskip 6pt
 Conditions
$(2)$ and $(3)$ ensure that $s^{-r}\Theta_{S,T}(s)$ is holomorphic
at $s=0$. Since $K/k$ is an extension of totally real fields and $S$
contains all infinite places the second condition is satisfied by
default. The condition $(4)$ is easily satisfied, for example if $T$
contains primes of two different residue characteristics.\vskip 6pt
For any set $V$ of places of $k$, we denote by $V_{K}$ the set of
places of $K$ lying above places in $V$ and by $\mathbb{Z}V_{K}$ the
free abelian group on $V_{K}$. Let $M$ be a $\mathbb{Z}$-module. If
$\mathbf{R}$ is one of the fields $\mathbb{Q}, \mathbb{R}$ or
$\mathbb{C}$, we denote by $\mathbf{R}M$ the tensor product
$\mathbf{R}\otimes_{\mathbb{Z}} M$. We extend this notation to sets
of primes of $K$, we denote by $\mathbf{R}V_{K}$ the tensor product
$\mathbf{R}\otimes_{\mathbb{Z}}\mathbb{Z}V_{K}$. The exterior power
over $\mathbb{Z}[G]$, and $\mathrm{Hom}$ of $\mathbb{Z}[G]$-modules
are denoted by
\begin{equation*}
\bigwedge_{G}\;,\; \mathrm{Hom}_{G}(-,-)
\end{equation*}
respectively.\vskip 6pt
 Assume that $V$ is finite and contains only finite primes. We denote by $S_{\infty}$  the
set of infinite places of
 $k$. Let $S=S_{\infty}\cup V$, so that
 \begin{equation*}
\mathbb{R}S_{K}=\mathbb{R}S_{\infty,K}\oplus\mathbb{R}V_{K}
\end{equation*} (as $\mathbb{R}[G]$-modules) and let $\pi_{\infty}$ denote
the projection from $\mathbb{R}S_{K}$ to $\mathbb{R}S_{\infty,K}$.
We define $\mathcal{L}_{S,\infty}$ as the composite
$\pi_{\infty}\circ \mathcal{L}_{S}$:
\begin{equation}\label{the map LS}
\xymatrix@=2pc{\mathcal{L}_{S,\infty}:
\mathbb{R}U_{S,T}(K)\ar[r]^-{\mathcal{L}_{S}}&
\mathbb{R}S_{K}\ar[r]^-{\pi_{\infty}}& \mathbb{R}S_{\infty,K}},
\end{equation}
where $\mathcal{L}_{S}$ is a logarithmic 'embedding' of
$U_{S,T}(K)$:
\begin{equation*}
\begin{array}{cccc}
\mathcal{L}_{S}: & U_{S,T}(K) & \xymatrix@=1.5pc{\ar[r]&} & \mathbb{R}S_{K} \\
 & \varepsilon & \xymatrix@=1.5pc{\ar@{|->}[r]&} & -\sum_{w\in
 S_{K}}\log(|\varepsilon|_{w})w.
\end{array}
\end{equation*}
Taking $r$-th exterior powers over the commutative ring
$\mathbb{R}[G]$ gives an $\mathbb{R}[G]$-linear map
\begin{equation*}
\xymatrix@=2pc{
\bigwedge^{r}_{\mathbb{R}[G]}\mathcal{L}_{S,\infty}:\;\bigwedge^{r}_{\mathbb{R}[G]}U_{S,T}(K)\ar[r]&
\bigwedge^{r}_{\mathbb{R}[G]}\mathbb{R}S_{\infty,K}=\mathbb{R}[G](w_{1}\wedge...\wedge
w_{r})},
\end{equation*}
where $w_{1},...,w_{r}$ is a choice of $r$-places of $K$ above the
infinite places $\{v_{1},...,v_{r}\}$ of $k$. Since
$w_{1}\wedge...\wedge w_{r}$ is a free generator we can define a
unique $\mathbb{R}[G]$-linear 'regulator' $\mathrm{R}_{w}$, called
Rubin-Stark regulator:
\begin{equation*}
\mathbb{R}\bigwedge^{r}_{G}U_{S,T}(K) \longrightarrow
\mathbb{R}[G]\;\mbox{by}\;
\bigwedge^{r}_{\mathbb{R}[G]}\mathcal{L}_{S,\infty}(x)=\mathrm{R}_{w}(x)(w_{1}\wedge...\wedge
w_{r}).
\end{equation*}
Explicitly, every element of
$\mathbb{R}\bigwedge^{r}_{\mathbb{Z}[G]}U_{S,T}(K)$ is a finite sum
of terms of  the form
$\varepsilon_{1}\wedge\cdots\wedge\varepsilon_{r}$
 with $\varepsilon_{i}\in \mathbb{R}U_{S,T}(K)$ and
\begin{equation*}
\begin{array}{ccc}
 \mathrm{R}_{w}\;: \mathbb{R}\displaystyle{\bigwedge^{r}_{G}}U_{S,T}(K) & \xymatrix@=2pc{ \ar[r]&} & \mathbb{R}[G] \\
  \varepsilon=\varepsilon_{1}\wedge\cdots\wedge\varepsilon_{r} & \xymatrix@=2pc{\ar@{|->}[r]&} &
  \det(-\sum_{\sigma\in
  G}\log|\varepsilon_{i}^{\sigma}|_{w_{j}}\sigma^{-1})_{i,j=1}^{r}.
\end{array}
\end{equation*}
\begin{deft}\label{Rubin lattice}
For a finitely generated $G$-module $M$ and $r\in \mathbb{Z}_{\geq
0}$, we define Rubin's lattice by
\begin{equation*}
\bigcap_{G}^{r}M=\{ m\in
\mathbb{Q}\displaystyle{\bigwedge^{r}_{G}}M\,|\, \Phi(m)\in
\mathbb{Z}[G]\,\mbox{for all}\, \Phi\in
\displaystyle{\bigwedge_{G}^{r}\mathrm{Hom}_{G}(M,\mathbb{Z}[G])}\}.
\end{equation*}
\end{deft}
\begin{rem}\label{remark rubin lattice}
Let $M^{\prime}$ be a finitely generated $G$-module. If
$M\longrightarrow M^{\prime}$ is a $G$-homomorphism, then it induces
a natural $G$-homomorphism
\begin{equation*}
\xymatrix@=2pc{\displaystyle{\bigcap_{G}^{r}}M\ar[r]&
\displaystyle{\bigcap_{G}^{r}}M^{\prime}}.
\end{equation*}
Besides, if $M\longrightarrow M^{\prime}$ is injective and its
cokernel is torsion-free, then the induced map
\begin{equation*}
\xymatrix@=2pc{\displaystyle{\bigcap_{G}^{r}}M\ar[r]&
\displaystyle{\bigcap_{G}^{r}}M^{\prime}}.
\end{equation*}
is injective ($\mathrm{e.g.}$\, \cite[Lemma 2.11]{Sano}).\\
Note that the Sinnott index
$(\displaystyle{\bigcap^{s}_{G}}M:\widetilde{\displaystyle{\bigwedge^{s}_{G}}}M)$
is finite ( $\mathrm{ e.g.}$ \cite[Proposition 1.2]{Rubin96}), where
$\widetilde{\displaystyle{\bigwedge_{G}^{s}}} M$ denotes the image
of $\displaystyle{\bigwedge_{G}^{s}}M$ via the canonical morphism
\begin{equation*}
\xymatrix@=2pc{\displaystyle{\bigwedge_{G}^{s}}M\ar[r]&
\mathbb{Q}\displaystyle{\bigwedge_{G}^{s}}M}.
\end{equation*}
\end{rem}
 Let $\Theta_{S,T}^{(r)}(0)$ be the coefficient of
$s^{r}$ in the Taylor series of $\Theta_{S,T}$\,;
\begin{equation*}
\Theta_{S,T}^{(r)}(0):=\displaystyle{\lim_{x\rightarrow 0}}
s^{-r}\Theta_{S,T}^{(r)}(s).
\end{equation*}
 Conjecture $\mathrm{B}^{\prime}$ (Rubin-Stark conjecture) of \cite{Rubin96} predicts the existence of
certain elements
\begin{equation*}
\eta_{K,S,T}\in\displaystyle{\bigcap_{G}^{r}U_{S,T}(K)}\quad\mbox{such
that}\quad \mathrm{R}_{w}(\eta_{K,S,T})=\Theta_{S,T}^{(r)}(0).
\end{equation*}
\vskip 6pt Let $\mathfrak{f}$ denote the finite part of the
conductor of $K/k$ ( we assume that  $\mathfrak{f}\neq (1)$ ). For
any ideal $\mathfrak{a}$ we denote the product of all distinct prime
ideals dividing $\mathfrak{a}$ by $\widehat{\mathfrak{a}}$ and
$\mathrm{T}_{\mathfrak{a}}(K)$ the subgroup of $G$ generated by the
inertia groups $I_{\mathfrak{q}}(K/k)$ with $\mathfrak{q}\mid
\mathfrak{a}$. If $\mathfrak{a}=(1)$ we set
$\mathrm{T}_{(1)}=\{1\}$. For any cycle $\mathfrak{g}\mid
\widehat{\mathfrak{f}}$, we denote the maximal subextension of $K$
whose conductor is prime to $\mathfrak{f}\mathfrak{g}^{-1}$ by
$K_{\mathfrak{g}}=K^{I_{\mathfrak{f}\mathfrak{g}^{-1}}}$. In the
sequel, we will fix  a finite set $S^{\prime}$ of finite places of
$k$ which contains at least one finite place, and will denote by
$S_{\mathfrak{g}}$ the set
\begin{equation*}
S_{\mathfrak{g}}=S_{\infty}\cup \{\mathfrak{q}: \mathfrak{q}\mid
\mathfrak{g}\}\cup S^{\prime}.
\end{equation*}
Let us also denote by $S$ the set $S=S_{\widehat{\mathfrak{f}}}$.
\\ Since $K_{\mathfrak{g}}$ is totally real then the hypothesis
$\mathbf{H}(K_{\mathfrak{g}}/k,S_{\mathfrak{g}},T,r)$ is
satisfied.\vskip 6pt In the rest of this paper  we assume the
validity of Rubin-Stark conjecture.
\begin{deft}
We denote by $\mathrm{St}_{K,T}$ the $\mathbb{Z}[G]$-module
generated by $\eta_{K_{\mathfrak{g}},S_{\mathfrak{g}},T}$ for all
$\mathfrak{g}\mid \widehat{\mathfrak{f}}$.
 \end{deft}
We will see that
\begin{equation*}
\xymatrix@=2pc{\displaystyle{\bigcap_{\mathrm{Gal}(K_{\mathfrak{g}}/k)}^{r}}U_{S_{\mathfrak{g}},T}(K_{\mathfrak{g}})\ar@{^{(}->}[r]&
\displaystyle{\bigcap_{G}^{r}}U_{S,T}(K)}
\end{equation*}
(see remark \ref{remark rubin lattice}), which justifies our
definition.\vskip 6pt Recall that a $\mathbb{Z}[G]$-lattice is
 a finitely generated $\mathbb{Z}[G]$-module which is a
torsion-free $\mathbb{Z}$-module.\vskip 6pt
 Let $e_{S,r}:=\Sigma_{\chi\in\hat{G},r_{S}(\chi)=r}e_{\chi}$. Note
that $e_{S,r}\in \mathbb{Q}[G]$ and for any $\mathbb{Z}[G]$-lattice
$M$, the $\mathbb{Z}[G]$-module
\begin{equation*}
e_{S,r}M=\{e_{S,r}m, m\in M\}
\end{equation*}
is a  lattice of the $\mathbb{Q}$-vector space
$e_{S,r}(\mathbb{Q}M)$. \vskip 6pt The goal  of this paper is the
following theorem
\begin{theo}\label{theorem 1} The Sinnott index $(e_{S,r}\displaystyle{\bigcap^{r}_{G}}
U_{S,T}(K):e_{S,r}\mathrm{St}_{K,T})$ is finite, and we have
\begin{equation*}
[e_{S,r}\displaystyle{\bigcap^{r}_{G}}
U_{S,T}(K):e_{S,r}\mathrm{St}_{K,T}]=h_{K}.(e_{S,r}\mathbb{Z}[G]
:e_{S,r}U^{(r)}_{K}).(e_{S,r}\displaystyle{\bigcap_{G}^{r}}U_{S,T}(K):
e_{S,r}\widetilde{\displaystyle{\bigwedge_{G}^{r}}}U_{S,T}(K)).
\beta_{K}.
\end{equation*}
 where $U^{(r)}_{K}$ is the Sinnott module (see Definition \ref{Sinnot
 modul}) and $\beta_{K}$  is well determined, see $(\ref{beta})$.
\end{theo}
  \section{Image by the Rubin-Stark regulator}
  Throughout this section, let $F=K_{\mathfrak{g}}$, $\eta_{F}=\eta_{K_{\mathfrak{g}},S_{\mathfrak{g}},T}$
   the Rubin-Stark element  in $K_{\mathfrak{g}}$. Let $H$ (resp. $\Delta$) denote the
   Galois group $\mathrm{Gal}(K/F)$ (resp. $\mathrm{Gal}(F/k)$). Let
   \begin{equation*}
   \xymatrix@=2pc{ \pi_{F}:\; \mathbb{C}[G]\ar[r]&
   \mathbb{C}[\Delta]}
   \end{equation*} denote the homomorphism induced by the natural
   surjection $\xymatrix@=1pc{ G\ar@{->>}[r]&\Delta}$, and let us
   fix $\gamma_{1},\cdots,\gamma_{d}\in G$, such that
   \begin{enumerate}
    \item $\gamma_{1}=1$
    \item $\{\pi_{F}(\gamma_{1}),\cdots,
    \pi_{F}(\gamma_{d})\}=\Delta.$
   \end{enumerate}
  \begin{pro}\label{Proposition regulator}
  Let $R_{w^{'}}$ be the restriction of the regulator map $R_{w}$
  to the subfield $F$ defined by using the infinites places
  $w_{1}^{'},..,w_{r}^{'}$ of $F$ below the places $w_{1},..,w_{r}$ of $K$.
 Then for any element $u_{F}\in
 \mathbb{R}\displaystyle{\bigwedge^{r}_{\Delta}}U_{S,T}(F)$ we have
\begin{equation*}
\pi_{F}(R_{w}(u_{F}))=|H|^{r}R_{w^{'}}(u_{F}).
\end{equation*}
\end{pro}
\begin{proof}  By definition
\begin{equation*}
 R_{w}(u_{F})=\mathrm{det}(a_{i,j})_{i,j},
 \end{equation*}
  where
\begin{equation*}
   a_{i,j}=-\Sigma_{\sigma\in
   G}\mathrm{log}\mid(u_{F})_{i}^{\sigma^{-1}}\mid_{w_{j}}\sigma,
   \end{equation*}
  here we denote by $u_{F}=(u_{F})_{1}\wedge..\wedge(u_{F})_{i}
  \wedge..\wedge(u_{F})_{r}$.
  Let us first calculate the coefficient $a_{i,j}$ for some given $(i,j)$.
   To simplify notations we refer to $(u_{F})_{i}$ simply as $u$. Then
  \begin{align*}
 \pi_{F}( \Sigma_{\sigma\in G}\mathrm{log}\mid u^{\sigma^{-1}}\mid_{w_{j}}\sigma)&={\Sigma}_{i=1}^{d}
  {\Sigma}_{h\in H}\mathrm{log}\mid u^{\gamma_{i}^{-1}h^{-1}}\mid_{w_{j}}\pi_{F}(\gamma_{i}h)\\
  &={\Sigma}_{i=1}^{d}{\Sigma}_{h\in H}\mathrm{log}\mid u^{\pi_{F}(\gamma_{i})^{-1}}\mid_{w_{j}}\pi_{F}(\gamma_{i})\;,\quad(u\in\mathbb{R}U_{S,T}(F))\\
  &=|H|(\Sigma_{i=1}^{d}\mathrm{log}\mid
  u^{\pi_{F}(\gamma_{i})^{-1}}\mid_{w_{j}}\pi_{F}(\gamma_{i})).
  \end{align*}
  Since $w^{\prime}_{j}={w_{j}}_{\mid F}$ is  completely decomposed in
  $K/F$, we obtain $\mid u^{\gamma_{i}^{-1}}\mid_{w_{j}}=\mid
  u^{\gamma_{i}^{-1}}\mid_{w^{\prime}_{j}}$.
 Finally we have
\begin{equation*}
 \pi_{F}(\mathrm{R}_{w}(u_{F}))=|H|^{r}\mathrm{R}_{w^{'}}(u_{F})
 \end{equation*}
 where $R_{w^{'}}$ is the same as $R_{w}$ but defined over
 $F$ instead of $K$ using the infinite places $w_{1}^{'},..,w_{r}^{'}$
 of $F$ below the places $w_{1},..,w_{r}$ of $K$.
 \end{proof}
For any character $\psi\in \hat{\Delta}$, let $\mathfrak{f}_{\psi}$
denote the conductor of $\psi$. Let $\hat{\psi}$ denote the
associated primitive character obtained by restricting $\psi$ to
$\Delta/\mathrm{ker}(\psi)$ (so that we obtain a faithful
character). Let us denote by $L(s,\widehat{\psi})$ the primitive
Hecke $L$-function defined for
  $ \mathrm{Re}(s)> 1$ by the Euler product
  \begin{equation*}
  L(s,\widehat{\psi})=\prod_{\mathfrak{p}\nmid
  \mathfrak{f}_{\psi}}(1-\widehat{\psi}(\sigma_{\mathfrak{p}})N\mathfrak{p}^{-s})^{-1}.
  \end{equation*}
   The function
$L(s,\widehat{\psi})$ can be analytically continued to an analytic
function on $\mathbb{C}$ (meromorphic when $\psi=1$). For any $s\in
\mathbb{C}$ and any non trivial character $\psi$ we have
\begin{equation*}
L_{S}(s,\psi)=\prod_{\substack{\mathfrak{p}\mid \mathfrak{f}_{F}\\
\mathfrak{p}\nmid
\mathfrak{f}_{\psi}}}(1-\widehat{\psi}(\sigma_{\mathfrak{p}})N\mathfrak{p}^{-s})
L(s,\widehat{\psi})
\end{equation*}
where  $\mathfrak{f}_{F}$ is the conductor of $F/k$. Since $F/k$ is
an extension of totally real fields, we have
\begin{equation*}
\mathrm{ord}_{s=0}(L_{S,T}(s,\psi))=\mathrm{ord}_{s=0}(L_{S}(s,\psi))=\mathrm{ord}_{s=0}(L_{S}(s,\widehat{\psi})).
\end{equation*}
Then
\begin{eqnarray*}
  L^{(r)}_{S,T}(0,\psi) &=& L^{(r)}_{S}(0,\psi).\prod_{\mathfrak{q}\in T}(1-\psi(\sigma_{\mathfrak{q}})\mathrm{N}\mathfrak{q}) \\
   &=&\prod_{\mathfrak{q}\in T}(1-\psi(\sigma_{\mathfrak{q}})\mathrm{N}\mathfrak{q}).\prod_{\substack{\mathfrak{p}\mid \mathfrak{f}_{F}\\
\mathfrak{p}\nmid
\mathfrak{f}_{\psi}}}(1-\widehat{\psi}(\sigma_{\mathfrak{p}}))
L^{(r)}(0,\widehat{\psi}).
\end{eqnarray*}
Remark that for any prime $\mathfrak{p}$ we have
\begin{equation*}
\sigma_{\mathfrak{p}}^{-1}e_{I_{\mathfrak{p}}}e_{\psi^{-1}}=\hat{\psi}(\sigma_{\mathfrak{p}})e_{\psi^{-1}}
\end{equation*}
where
$e_{I_{\mathfrak{p}}}=\frac{1}{|I_{\mathfrak{p}}|}\displaystyle{\sum_{\sigma\in
I_{\mathfrak{p}}}}\sigma$. Hence we have the following proposition
\begin{pro}\label{Proposition omega}
 There exists an element $\omega_{K}\in \mathbb{C}[G]$
independent of the choice of the field $F$ which verifies
\begin{equation*}
\pi_{F}(e_{S,r})R_{w^{'}}(\eta_{F})=\pi_{F}\big(e_{S,r}\omega_{K}(\delta_{T}\prod_{\mathfrak{p}\mid
\mathfrak{f}_{F}}(1-\sigma_{\mathfrak{p}}^{-1}e_{I_{\mathfrak{p}}}))\big),
  \end{equation*}
  where
\begin{equation*}
   \omega_{K}:=\displaystyle{
 \sum_{\chi\in\widehat{G},\,
 r_{S}(\chi)=r}}L^{(r)}(0,\widehat{\chi})e_{\chi^{-1}}\quad\mbox{and}\quad
 \delta_{T}:=\prod_{\mathfrak{q}\in
 T}(1-\sigma_{\mathfrak{q}}^{-1}\mathbf{N}\mathfrak{q})
 \end{equation*}
 \end{pro}
 \begin{proof}
 As we previously stated
 \begin{equation*}
 R_{w^{'}}(\eta_{F})=\Theta^{(r)}_{S,T,F/k}(0
 )=\Sigma_{\psi\in \hat{\Delta}}L_{S,T}^{(r)}(0,\psi)e_{\psi^{-1}}.
 \end{equation*}
Since $L^{(r)}_{S,T}(0,\psi)=\displaystyle{\prod_{\mathfrak{q}\in
T}(1-\psi(\sigma_{\mathfrak{q}})\mathrm{N}\mathfrak{q}).\prod_{\mathfrak{p}\mid
\mathfrak{f}_{F},\, \mathfrak{p}\nmid
\mathfrak{f}_{\psi}}}(1-\widehat{\psi}(\sigma_{\mathfrak{p}}))
L^{(r)}(0,\widehat{\psi})$ holds, we obtain
\begin{align*}
\mathrm{R}_{w^{'}}(\eta_{F})&=\Sigma_{\psi\in \hat{\Delta},
r_{S}(\psi)=r} \big((\prod_{\mathfrak{q}\in
T}(1-\psi(\sigma_{\mathfrak{q}})\mathbf{N}\mathfrak{q}))({\prod}_{\mathfrak{p}\mid
\mathfrak{f}_F, \mathfrak{p}\nmid
\mathfrak{f}_{\psi}}(1-\hat{\psi}(\sigma_{\mathfrak{p}}))\big)
L^{(r)}(0,\hat{\psi})e_{\psi^{-1}}\\
&=\Sigma_{\psi\in \hat{\Delta},
r_{S}(\psi)=r}\big((\prod_{\mathfrak{q}\in
T}(1-\sigma_{\mathfrak{q}}^{-1}\mathbf{N}\mathfrak{q}))({\prod}_{\mathfrak{p}\mid
\mathfrak{f}_F,
 \mathfrak{p}\nmid \mathfrak{f}_{\psi}}(1-\sigma_{\mathfrak{p}}^{-1}e_{I_{\mathfrak{p}}}))e_{\psi^{-1}}\big)
 (L^{(r)}(0,\hat{\psi})e_{\psi^{-1}})\\
&=\big({\Sigma}_{\psi\in \hat{\Delta},
r_{S}(\psi)=r}(\prod_{\mathfrak{q}\in
T}(1-\sigma_{\mathfrak{q}}^{-1}\mathbf{N}\mathfrak{q}).{\prod}_{\mathfrak{p}\mid
\mathfrak{f}_F, \mathfrak{p}\nmid
\mathfrak{f}_{\psi}}(1-\sigma_{\mathfrak{p}}^{-1}e_{I_{\mathfrak{p}}}))e_{\psi^{-1}}\big)\big({\Sigma}_{\psi\in
\hat{\Delta}, r_{S}(\psi)=r}L^{(r)}(0,\hat{\psi})e_{\psi^{-1}}\big)
\end{align*}
where $I_{\mathfrak{p}}$ is the inertia group of $\mathfrak{p}$ in
$F/k$. Using the fact that each character of
$\Delta=\mathrm{Gal}(F/k)$ can be seen as a character of
$G=\mathrm{Gal}(K/k)$ trivial on $H=\mathrm{Gal}(K/F)$, we get
\begin{center}
$\pi_{F}(e_{\psi^{-1}\circ \pi_{F}})=e_{\psi^{-1}}$\;\; and\;\;
$\sigma_{\mathfrak{p}}^{-1}e_{I_{\mathfrak{p}}}e_{\psi^{-1}\circ
\pi_{F}}=\hat{\psi}(\sigma_{\mathfrak{p}})e_{\psi^{-1}\circ\pi_{F}}$
\end{center}
 where $I_{\mathfrak{p}}$ denotes also the inertia group of
$\mathfrak{p}$ in $K/k$. Therefore
\begin{equation*}
 \pi_{F}(e_{S,r})R_{w^{\prime}}(\eta_{F})=\pi_{F}\bigg( e_{S,r}\big(\Sigma_{\substack{\chi\in
  \hat{G},\,
r_{S}(\chi)=r\\\chi(H)=1}}(\prod_{\mathfrak{q}\in
T}(1-\sigma_{\mathfrak{q}}^{-1}\mathbf{N}\mathfrak{q}))(\prod_{\mathfrak{p}\mid
\mathfrak{f}_F, \mathfrak{p}\nmid
\mathfrak{f}_{\chi}}(1-\sigma_{\mathfrak{p}}^{-1}e_{I_{\mathfrak{p}}}))e_{\chi^{-1}}\big)\omega_{K}\bigg)
\end{equation*}
where
\begin{equation*}
\omega_{K}:={\Sigma}_{\chi\in \hat{G},\,
r_{S}(\chi)=r}L^{(r)}(0,\hat{\chi})e_{\chi^{-1}}.
\end{equation*}
Since
\begin{equation*}
 \pi_{F}(e_{\chi})=\left\{
                           \begin{array}{ll}
                             0, & \hbox{if $\chi(H)\neq 1$;} \\
                             \frac{1}{|\Delta|}\sum_{\sigma\in \Delta}\chi(\sigma)\sigma^{-1}, & \hbox{if $\chi(H)=1$.}
                           \end{array}
                         \right.
\end{equation*}
holds, we get
 \begin{equation*}
\pi_{F}(e_{S,r})\mathrm{R}_{w^{'}}(\eta_{F})=\pi_{F}\bigg(e_{S,r}
\omega_{K}\big(\delta_{T}.\prod_{\mathfrak{p}\mid
\mathfrak{f}_F}(1-\sigma_{\mathfrak{p}}^{-1}e_{I_{\mathfrak{p}}})\big)\bigg)
 \end{equation*}
 where $\delta_{T}=\prod_{\mathfrak{q}\in
 T}(1-\sigma_{\mathfrak{q}}^{-1}\mathbf{N}\mathfrak{q})$. This finishes the proof of the proposition.
\end{proof}
We combine the results of the two previous sections and get
\begin{coro}\label{corollary 3}
Recall that $H:=\mathrm{Gal}(K/F)$. Then
\begin{equation*}
\pi_{F}(e_{S,r}\mathrm{R}_{w}(\eta_{F}))=\pi_{F}\big(\omega_{K}(|H|^{r}(\delta_{T}.\prod_{\mathfrak{p}\mid
\mathfrak{f}_F}(1-\sigma_{p}^{-1}e_{I_{\mathfrak{p}}}))e_{S,r}\big).
\end{equation*}
\end{coro}
\section{Index of the "Stark" module}
\subsection{The generalised Sinnott index}
 We recall some data about the generalised Sinnott index. For a more complete exhibit
  of the properties of this index the reader is invited to refer to \cite{Sinn1}.
  Let $p$ be a prime rational and $v_{p}$ its normalised valuation
($v_{p}(p)=1$). Let $\mathbb{F}$ be one of the fields $\mathbb{Q}$,
$\mathbb{Q}_p$ or $\mathbb{R}$, and let
\begin{equation*}
\mathcal{O}:= \left\{
  \begin{array}{ll}
    \mathbb{Z}, & \hbox{$\mathbb{F}=\mathbb{Q}\;\mbox{or}\;\mathbb{R}$;} \\
    \mathbb{Z}_{p}, & \hbox{$\mathbb{F}=\mathbb{Q}_{p}$.}
  \end{array}
\right.
\end{equation*}
Let $E$ be an $\mathbb{F}$-vector space of finite dimension $d$. An
$\mathcal{O}$-lattice $\Lambda$ is a free $\mathcal{O}$-submodule of
$E$ of rank $d$ such that the $\mathbb{F}$-vector space generated by
$\Lambda$ is $E$. If $M$ and $N$ are two lattices of $E$, we define
the generalised Sinnott index as follows
\begin{equation*}
 (M:N)=\left\{
\begin{array}{lcl}
\mid\mathrm{det}(\gamma)\mid\;\mathrm{if}\; \mathbb{F}=\mathbb{Q}\;\mathrm{or}\;\mathbb{R} \\
p^{v_{p}(\mathrm{det}(\gamma))} \;\mathrm{if}\; \mathbb{F}=\mathbb{Q}_p \\
\end{array}
\right.
\end{equation*}
where $\gamma$ is an automorphism of the $\mathbb{F}$-vector space
$E$ such that $\gamma (M)=N$.\vskip 6pt Recall that
$\mathrm{T}_{\mathfrak{r}}(K)$ denotes the subgroup of $G$ generated
by the inertia groups $I_{\mathfrak{q}}(K/k)$ with $\mathfrak{q}\mid
\mathfrak{r}$.
\begin{deft}\label{Sinnot modul}
Let $\mathfrak{f}$ be the conductor of $K/k$. Let $\mathfrak{s}$ be
a divisor of $\widehat{\mathfrak{f}}$. If $\mathfrak{s}\neq (1)$,
then we denote by $U^{(r)}_{\mathfrak{s}}$ or
$U^{(r)}_{\mathfrak{s},K}$ the
$\mathbb{Z}[\mathrm{Gal}(K/k)]$-submodule of
$\mathbb{Q}[\mathrm{Gal}(K/k)]$ generated by all the elements
\begin{equation*}
\alpha(\mathfrak{r},\mathfrak{s})=s(\mathrm{T}_{\mathfrak{r}}(K))^{r}\prod_{\mathfrak{p}\mid
\mathfrak{s}/\mathfrak{r}}(1-\sigma_{\mathfrak{p}}^{-1}e_{I_{\mathfrak{p}}});\;
\mathfrak{r}\mid \mathfrak{s},\;\mbox{where}\;
s(\mathrm{T}_{\mathfrak{r}}(K))=\sum_{\sigma\in\mathrm{T}_{\mathfrak{r}}(K)}\sigma.
\end{equation*}
Moreover we set $U^{(r)}_{(1)}=\mathbb{Z}[\mathrm{Gal}(K/k)]$,
$U^{(r)}_{K}=U^{(r)}_{\widehat{\mathfrak{f}}}$ and
$U^{(1)}_{\mathfrak{s}, K}=U_{\mathfrak{s}}$.
\end{deft}
\begin{rem} The modules $U_{\mathfrak{s}}$   were introduced in \cite{Sinn1} when $k$ is equal to the
field of rational numbers $\mathbb{Q}$. Sinnott used these modules
to study the index of cyclotomic units in the cyclotomic
$\mathbb{Z}_{p}$-extension. This technique has been followed in the
case of circular units or in \cite{Oukh} for the elliptic units
case.
\end{rem}
\begin{lem}\label{Lemme indeces}
The following generalized Sinnott indices are well defined
\begin{enumerate}
\item $(e_{S,r}U^{(r)}_{K}: e_{S,r}\omega_{K}U^{(r)}_{K})$
\item $(e_{S,r}\mathbb{Z}[G]:e_{S,r}U^{(r)}_{K})$
\item $(e_{S,r}\mathbb{Z}[G]:\mathrm{R}_{w}(e_{S,r}\displaystyle{\bigcap^{r}_{G}}U_{S,T}(K))$
\end{enumerate}
\end{lem}
\begin{proof} The assertions $(1)$ and $(2)$ are  a direct consequence
of the fact that $U^{(r)}_{K}$ is a lattice of $\mathbb{Q}[G]$ and
the definition of the generalized Sinnott index. The image of
$e_{S,r}\displaystyle{\bigcap^{r}_{G}} U_{S,T}(K)$ by the
Rubin-Stark regulator is a lattice of $e_{S,r}\mathbb{Q}[G]$ and
hence, the index in $(3)$ is well defined.
\end{proof}
\begin{coro}\label{Corollary Sinnot}
The generalized Sinnott index
$(e_{S,r}\displaystyle{\bigcap^{r}_{G}}
U_{S,T}(K):e_{S,r}\mathrm{Stark}_{K,T})$ is well defined and we have
the equality \begin{equation*}
(e_{S,r}\displaystyle{\bigcap^{r}_{G}}
 U_{S,T}(K):e_{S,r}\mathrm{Stark}_{K,T})=\frac{(e_{S,r}
 \mathbb{Z}[G]:e_{S,r}\delta_{T}U^{(r)}_{K})}{(e_{S,r}\mathbb{Z}[G]
 :\mathrm{R}_{w}(e_{S,r}\displaystyle{\bigcap^{r}_{G}} U_{S,T}(K)))}.(e_{S,r}U^{(r)}_{K}:
\omega_{K} e_{S,r}U^{(r)}_{K}).
\end{equation*}
\end{coro}
\begin{proof}
 The index $(R_{w}(e_{S,r}\displaystyle{\bigcap^{r}_{G}}
U_{S,T}(K)):R_{w}(e_{S,r}\mathrm{Stark}_{K,T}))$ is well defined and
the map $\mathrm{R}_{w}$ is injective, thus
\begin{equation*}
(e_{S,r}\displaystyle{\bigcap^{r}_{G}}
U_{S,T}(K):e_{S,r}\mathrm{Stark}_{K,T})=
(\mathrm{R}_{w}(e_{S,r}\displaystyle{\bigcap^{r}_{G}}
U_{S,T}(K)):\mathrm{R}_{w}(e_{S,r}\mathrm{Stark}_{K,T})).
\end{equation*}
Since
$\mathrm{R}_{w}(e_{S,r}\mathrm{Stark}_{K})=\omega_{K}e_{S,r}\delta_{T}U^{(r)}_{K}$
(see Corollary \ref{corollary 3}) holds, we obtain
\begin{equation*}
\begin{array}{c}
  (R_{w}(e_{S,r}\displaystyle{\bigcap^{r}_{G}}
U_{S,T}(K)):\mathrm{R}_{w}(e_{S,r}\mathrm{Stark}_{K,T}) ) \\
=\\\frac{(e_{S,r}\mathbb{Z}[G]:e_{S,r}\delta_{T}U^{(r)}_{K})}
{(e_{S,r}\mathbb{Z}[G]:\mathrm{R}_{w}(e_{S,r}\displaystyle{\bigcap^{r}_{G}}
U_{S,T}(K)))}(e_{S,r}\delta_{T}U^{(r)}_{K}:
\omega_{K}e_{S,r}\delta_{T}U^{(r)}_{K}).
\end{array}
\end{equation*}
Using the fact that $\delta_{T}=\prod_{\mathfrak{q}\in
T}(1-\sigma_{\mathfrak{q}}^{-1}\mathbf{N}\mathfrak{q})$ is a
non-zero-divisor, we get
\begin{equation*}
(e_{S,r}\delta_{T}U^{(r)}_{K}:
\omega_{K}e_{S,r}\delta_{T}U^{(r)}_{K})=(e_{S,r}U^{(r)}_{K}:
\omega_{K}e_{S,r}U^{(r)}_{K}).
\end{equation*}
Hence the corollary follows.
\end{proof}
\subsection{The class number Formula}
 Next, we use the previous result to prove the class number formula
 shown in Theorem \ref{theorem 1}.\vskip 6pt
  Let $F/k$ be an intermediate extension in $K/k$, we denote by
$\mathrm{Ram}(F/k)$ the set of primes that ramify in the extension
$F/k$. We make some further notations
\begin{enumerate}
\item $X(F):=\{ \Sigma a_{w}w\in \mathbb{Z}S_{\infty,F}, \Sigma a_{w}=0\}$.
\item $\lambda_{F}: U_{S_{\infty}}(F)\longrightarrow X(F)\otimes\mathbb{R}$ is the
 map defined by
  \begin{equation*}
  \lambda_{F}(\alpha)=-\Sigma_{w\in S_{\infty,F}}\mathrm{log}(\mid \alpha
  \mid_{w})w.
  \end{equation*}
\item $\mathrm{Reg}_{F}=\mid \mathrm{det}(\lambda_{F})\mid$ the regulator associated to $\lambda_{F}$.
\item We assume that $\mathrm{Ram}(K/k)=\{\mathfrak{p}_{1},..,\mathfrak{p}_{\mid Ram(K/k) \mid}\}$.
 For $I\subset \{1,..,\mid \mathrm{Ram}(K/k) \mid\}$ we define  the
field
\begin{equation*}
K_{I}:=K^{D_{I}}
\end{equation*}
 where $D_{I}$ is the subgroup of $G$ generated by the
decomposition groups $D_{i}$ of $\mathfrak{p}_{i}$ in
 $K/k$, $i\in I$.
\end{enumerate}
\begin{lem}\label{lemma eS}
One has $e_{S,r}U_{S,T}(K)=e_{S,r}U_{S_{\infty}}(K)$.
\end{lem}
\begin{proof} Let $S_{1}$ be a finite set of places of $K$, and let
$S_{2}=S_{1}\cup \{\mathfrak{q}_{v}\}$. Let $\{u_{1},\cdots,u_{t}\}$
be fundamental units of $\mathcal{O}^{\ast}_{S_{1}}$. We claim that
if $\mathfrak{q}_{v}^{m}=a\mathcal{O}_{S_{1}}$ then
$\{u_{1},\cdots,u_{t},a\}$ are  fundamental units for
$\mathcal{O}^{\ast}_{S_{2}}$, and $a^{1-e_{D_{v}}}\in
\mathcal{O}_{S_{1}}^{\ast}$, where $m$ is the order of
$\mathfrak{q}_{v}$ in the ideal class group of $\mathcal{O}_{S_{1}}$
,$D_{v}$ is the decomposition group of $\mathfrak{q}_{v}$ in $K/k$
and $e_{D_{v}}=\frac{1}{|D_{v}|}N_{D_{v}}$. First we prove that this
claim will give the desired result. Since $|S|>r+1$, we obtain
\begin{equation*}
e_{S,r}=\prod_{v\in S-S_{\infty}}(1-e_{D_{v}}).
\end{equation*}
Iterating our claim gives
\begin{equation*}
e_{S,r}U_{S,T}(K)\subset e_{S,r}U_{S_{\infty}}(K)
\end{equation*}
as desired.\vskip 6pt It remains to prove our claim that
$\{u_{1},\cdots,u_{t},a\}$ are  fundamental units for
$\mathcal{O}^{\ast}_{S_{2}}$. Let $u$ be a unit of
$\mathcal{O}_{S_{2}}$. By scaling by an appropriate power of $a$, we
may assume that $0\leq i=v_{\mathfrak{q}_{v}}(u)\leq m-1$. Then
$\mathfrak{q}_{v}^{m}=u\mathcal{O}_{S_{1}}$. Since the order of
$\mathfrak{q}_{v}$ in the ideal class group of $\mathcal{O}_{S_{1}}$
is $m$, we must have $i=0$, so that $u\in
\mathcal{O}^{\ast}_{S_{1}}$. Then  we have $
\mathfrak{q}_{v}^{1-e_{D_{v}}}=\mathcal{O}_{S_{1}}, $ and hence $
a^{1-e_{D_{v}}}\in \mathcal{O}_{S_{1}}^{\ast}. $
\end{proof}
Recall that for a $G$-module,
$\widetilde{\displaystyle{\bigwedge_{G}^{s}}} M$ denotes the image
of $\displaystyle{\bigwedge_{G}^{s}}M$ via the canonical morphism
\begin{equation*}
\xymatrix@=2pc{\displaystyle{\bigwedge_{G}^{s}}M\ar[r]&
\mathbb{Q}\displaystyle{\bigwedge_{G}^{s}}M}.
\end{equation*}
Using the properties of $\mathrm{det}$ and the fact that the
category of $\mathbb{Q}[G]$-modules is semi-simple, we obtain the
following lemma
\begin{lem} Let $M$ and $N$ be $\mathbb{Z}[G]$-lattices, such that
the Sinnott index $(M:N)$ is defined. Then, we have
\begin{equation*}
(M:N)=(\widetilde{\displaystyle{\bigwedge_{G}^{s}}} M:
\widetilde{\displaystyle{\bigwedge_{G}^{s}}} N),
\end{equation*}
where $s$ is maximal.
\end{lem}
\begin{proof}
Exercise .
\end{proof}
\begin{deft}
Let $M$ be a $\mathbb{Z}[G]$-lattice. We denote by $S(M)$ the
semi-simplified  of $M$. It is the smallest module completely
decomposable containing $M$, and definite by
\begin{equation*}
S(M):=\bigoplus_{\chi\in \mathcal{X}}e_{\chi}M\subset \mathbb{Q}M
\end{equation*}
where $\mathcal{X}$ is the set of all irreducible characters of $G$
over $\mathbb{Q}$.
\end{deft}
Note that the index of $M$ in $S(M)$ is finite. Indeed, let $g=|G|$.
Since $gS(M)\subset M$ and $M$ is a finitely generated module, we
get
\begin{equation*}
(S(M):M)\mid g^{\mathrm{rank}_{\mathbb{Z}}(M)}.
\end{equation*}
 To go further, we need some
notations. For any subextension $F$ of $K/k$, we put
\begin{equation}\label{defenition of cF}
c_{F}=\frac{(S(\lambda_{K}(U_{S_{\infty}}(K)^{N_{H}})):\lambda_{K}(U_{S_{\infty}}(K)^{N_{H}}))}
{(S(X(K)^{N_{H}}):X(K)^{N_{H}})}.
|\widehat{H}^{0}(H,U_{S_{\infty}}(K))|^{-1}
\end{equation}
and
\begin{equation}\label{definition of cesr}
c_{K,r}=\frac{(S(e_{S,r}\lambda_{K}
(U_{S}(K)):e_{S,r}\lambda_{K}(U_{S}(K))}{(S(e_{S,r}X(K)):e_{S,r}X(K))
}
\end{equation}
where $H=\mathrm{Gal}(K/F)$ and $N_{H}=\sum_{\sigma\in
H}\sigma$.\vskip 6pt The following proposition is crucial for our
purpose.
\begin{pro}\label{Proposition index of Regulator}
\begin{equation*}
(e_{S,r}\mathbb{Z}[G]:R_{w}(e_{S,r}\widetilde{\displaystyle{\bigwedge^{r}_{G}}}
U_{S,T}(K)))=
\mathrm{Reg}_{K}c_{K,r}.c_{K}^{-1}\prod_{I\subset\{1,\cdots,|\mathrm{Ram}(K/k)|\}}c_{K_{I}}^{(-1)^{|I|+1}}
\mathrm{Reg}_{K_{I}}^{(-1)^{|I|}}.
\end{equation*}
\end{pro}
\begin{proof} Let $S=S_{\infty}\cup V$ and let
$\mathcal{L}_{S,\infty}$ the map defined in $(\ref{the map LS})$.
The facts that $e_{S,r}\mathbb{R}V_{K}=0$ ($|S|>r+1$) and that the
map
\begin{equation*}
\xymatrix@=2pc{e_{S,r}\mathcal{L}_{S,\infty}:
e_{S,r}\mathbb{R}U_{S,T}(K)\ar[r]^-{e_{S,r}\mathcal{L}_{S}}&
e_{S,r}\mathbb{R}S_{K}\ar[r]^-{\mathrm{id}}&
e_{S,r}\mathbb{R}S_{\infty,K}:=e_{s,r}X(K)}
\end{equation*}
is an isomoprhism, show that
\begin{equation*}
(e_{S,r}X(K): e_{S,r}\mathcal{L}_{S,\infty}(e_{S,r}U_{S,T}(K)))
=\det(e_{S,r}\mathcal{L}_{S}).
\end{equation*}
Then, using the facts
\begin{eqnarray*}
  \big(e_{S,r}X(K):e_{S,r}
\mathcal{L}_{S,\infty}(e_{S,r}U_{S,T}(K))\big) &=&
\big(e_{S,r}\mathbb{Z}S_{\infty,K}:e_{S,r}
\mathcal{L}_{S,\infty}(e_{S,r}U_{S,T}(K))\big)\\
&=&  \big(e_{S,r}
\displaystyle{\widetilde{\bigwedge^{r}_{G}}}\mathbb{Z}S_{\infty,K}:e_{S,r}
\bigwedge^{r}\mathcal{L}_{S,\infty}(e_{S,r}\displaystyle{\widetilde{\bigwedge^{r}_{G}}}U_{S,T}(K))\big) \\
   &=&\big(e_{S,r}\mathbb{Z}[G](w_{1}\wedge...\wedge w_{r}):
   \mathrm{R}_{w}(e_{S,r}\displaystyle{\widetilde{\bigwedge^{r}_{G}}}U_{S,T}(K))
   (w_{1}\wedge...\wedge w_{r}) \big) \\
    &=&\big(e_{S,r}\mathbb{Z}[G]:\mathrm{R}_{w}(e_{S,r}\displaystyle{\widetilde{\bigwedge^{r}_{G}}}U_{S,T}(K)\big)\\
    &=&
    \big(e_{S,r}\mathbb{Z}[G]:R_{w}(e_{S,r}\displaystyle{\widetilde{\bigwedge^{r}_{G}}}
U_{S,T}(K))\big),
\end{eqnarray*}
 we obtain
$\big(e_{S,r}\mathbb{Z}[G]:R_{w}(e_{S,r}\displaystyle{\widetilde{\bigwedge^{r}_{G}}}
U_{S,T}(K))\big)=\det(e_{S,r}\mathcal{L}_{S})$. Therefore, using
lemma \ref{lemma eS}, we get
\begin{eqnarray*}
  (e_{S,r}X(K):e_{S,r}
\mathcal{L}_{S,\infty}(e_{S,r}U_{S,T}(K))) &=& (e_{S,r}X(K):e_{S,r}\lambda_{K}(U_{S_{\infty}}(K))  \\
   &=& c_{K,r}.(S(e_{S,r}X(K)):S(e_{S,r}\lambda_{K}(U_{S_{\infty}}(K)) \\
   &=&c_{K,r}.\prod_{\substack{\chi\in\widehat{G}\\
   r_{S}(\chi)=r}}(e_{\chi}X(K):e_{\chi}\lambda_{K}(U_{S_{\infty}}(K))
\end{eqnarray*}
  Let  $F$ be a subextension of $K/k$. On the one hand,
the commutative diagram
\begin{equation*}
\xymatrix@=2pc{
\mathbb{C}U_{S_{\infty}}(F)\ar[r]^-{\lambda_{F}}\ar[d]_-{i}^-{\wr}&
\mathbb{C}X(F)\ar[d]^-{j}_-{\wr}\\
\mathbb{C}U_{S_{\infty}}(K)^{H}\ar[r]^-{\lambda_{K}}&
\mathbb{C}X(K)^{H}}
\end{equation*}
shows that
\begin{equation*}
\mathrm{Reg}_{F}=(X(K)^{H}:\lambda_{K}(U_{S_{\infty}}(K)^{H})).(X(K)^{H}:j(X(F)))^{-1}.(U_{S_{\infty}}(K)^{H}:i(U_{S_{\infty}}(F)))
\end{equation*}
where
\begin{itemize}
    \item $j(w_{F}):=\sum_{w\mid w_{F}}[K_{w}:F_{w_{F}}]w=N_{H}w_{K}$,
    where $w_{K}\mid w_{F}$ is a place of $K$ laying above $w_{F}$
    \item $i(x)=x$.
\end{itemize}
Since $i$ is injective and $j(X(F))=N_{H}(X(K))$, we obtain
\begin{equation*}
\mathrm{Reg}_{F}=|\widehat{H}^{0}(H,X(K))|^{-1}.(X(K)^{H}:\lambda_{K}(U_{S_{\infty}}(K))^{H}).
\end{equation*}
Using the fact that $\xymatrix@=1pc{
U_{S_{\infty}}(K)\ar[r]^-{\lambda_{K}}& \mathbb{R}X(K)}$ is
injective as $G$-module, we get
\begin{equation*}
(X(K)^{H}:\lambda_{K}(U_{S_{\infty}}(K))^{H}).|\widehat{H}^{0}(H,U_{S_{\infty}}(K))|=
(X(K)^{N_{H}}:\lambda_{K}(U_{S_{\infty}}(K))^{N_{H}}).|\widehat{H}^{0}(H,X(K))|.
\end{equation*}
It follows that
\begin{eqnarray*}
  \mathrm{Reg}_{F} &=&|\widehat{H}^{0}(H,U_{S_{\infty}}(K))|^{-1}.(X(K)^{N_{H}}:\lambda_{K}(U_{S_{\infty}}(K))^{N_{H}})  \\
   &=&c_{F}.(S(X(K)^{N_{H}}):S(\lambda_{K}(U_{S_{\infty}}(K)^{N_{H}}))).
\end{eqnarray*}
where $c_{F}$ is defined in $(\ref{defenition of cF})$.
 On the other hand, for any $\widetilde{\chi}\in
\widehat{\mathrm{Gal}(F/k)}$, we have
\begin{eqnarray*}
 (e_{\widetilde{\chi}}X(K)^{N_{H}}:e_{\widetilde{\chi}}\lambda_{K}(U_{S_{\infty}}(K))^{N_{H}})  &=&
 (|H|e_{\widetilde{\chi}\circ \pi_{F}}X(K):|H|e_{\widetilde{\chi}\circ \pi_{F}}\lambda_{K}(U_{S_{\infty}}(K))\\
   &=& (e_{\widetilde{\chi}\circ \pi_{F}}X(K):e_{\widetilde{\chi}\circ
   \pi_{F}}\lambda_{K}(U_{S_{\infty}}(K)).
\end{eqnarray*}
Then
\begin{equation*}
\mathrm{Reg}_{F}=c_{F}.\prod_{\substack{\chi\in
\widehat{G}\\
\chi(\mathrm{Gal}(K/F))=1}}(e_{\chi}X(K):e_{\chi}\lambda_{K}(U_{S_{\infty}}(K)).
\end{equation*}
 Therefore, a simple inclusion-exclusion argument gives
\begin{equation*}
\prod_{\substack{\chi\in\widehat{G}\\
r_{S}(\chi)=r}}(e_{\chi}X(K):e_{\chi}\lambda_{K}(U_{S_{\infty}}(K)))=
c_{K}^{-1}\mathrm{Reg}_{K}\prod_{I\subset\{1,\cdots,|\mathrm{Ram}(K/k)|\}}c_{K_{I}}^{(-1)^{|I|+1}}
\mathrm{Reg}_{K_{I}}^{(-1)^{|I|}}
\end{equation*}
 Finally
\begin{equation*}
(e_{S,r}\mathbb{Z}[G]:R_{w}(e_{S,r}\widetilde{\displaystyle{\bigwedge^{r}_{G}}}
U_{S,T}(K)))=
c_{K,r}.c_{K}^{-1}\mathrm{Reg}_{K}\prod_{I\subset\{1,\cdots,|\mathrm{Ram}(K/k)|\}}c_{K_{I}}^{(-1)^{|I|+1}}
\mathrm{Reg}_{K_{I}}^{(-1)^{|I|}}.
\end{equation*}
\end{proof}
\noindent We prove now Theorem \ref{theorem 1}\vskip 6pt
\noindent\textbf{Theorem 1.4.} The Sinnott index
$(e_{S,r}\displaystyle{\bigcap^{r}_{G}}
U_{S,T}(K):e_{S,r}\mathrm{St}_{K,T})$ is finite, and we have
\begin{equation*}
[e_{S,r}\displaystyle{\bigcap^{r}_{G}}
U_{S,T}(K):e_{S,r}\mathrm{St}_{K,T}]=h_{K}.(e_{S,r}\mathbb{Z}[G]
:e_{S,r}U^{(r)}_{K}).(e_{S,r}\displaystyle{\bigcap_{G}^{r}}U_{S,T}(K):
e_{S,r}\widetilde{\displaystyle{\bigwedge_{G}^{r}}}U_{S,T}(K)).
\beta_{K}.
\end{equation*}
where \begin{equation*}
 \beta_{K}= c_{K}c_{K,r}^{-1}\prod_{I\subset
\{1,..,\mid
\mathrm{Ram}(K/k)\mid\}}c_{K_{I}}^{(-1)^{|I|}}h_{K_{I}}^{(-1)^{\mid
I \mid}}.
\end{equation*}
\begin{proof}
We begin by the expression obtained in Corollary \ref{Corollary
Sinnot} and analyse each term. We have
\begin{equation*}
(e_{S,r}U^{(r)}_{K}: \omega_{K}e_{S,r}U^{(r)}_{K})=\mid
\mathrm{det}(m_{\omega_{K}}) \mid
\end{equation*}
where $\omega_{K}:={\Sigma}_{\chi\in \hat{G},
r_{S}(\chi)=r}L^{(r)}(0,\hat{\chi})e_{\chi^{-1}}$ and
$m_{\omega_{K}}$ is the multiplication by $w_{K}$. Since the set $\{
e_{\chi},r_{S}(\chi)=r\}$ is an $\mathbb{R}$-base of the vector
space $e_{S,r}\mathbb{R}[G]$, and $e_{S,r}U^{(r)}_{K}$ is a lattice
of it,
\begin{equation*}
\mathrm{det}(m_{\omega_{K}})=\prod_{\chi\in \hat{G},
r_{S}(\chi)=r}L^{(r)}(0,\hat{\chi})
\end{equation*}
 A simple
inclusion-exclusion argument gives
\begin{equation*}
\prod_{\chi\in \hat{G}, r_{S}(\chi)=r}L^{(r)}(0,\hat{\chi}
)=\zeta^{*}_{K}(0)\prod_{I\subset \{1,..,\mid
\mathrm{Ram}(K/k)\mid\}}{\zeta^{*}_{K_{I}}(0)}^{(-1)^{\mid I \mid}}
\end{equation*}
 where $\zeta^{*}_{K_{I}}(0)$ is the first non trivial
term in the Taylor expansion of the function $\zeta_{K_{I}}(s)$ at
$0$ given by
\begin{equation*}
\zeta^{*}_{K_{I}}(0):=\mathrm{lim}_{s\to
0}s^{-\mathrm{ord}_{s=0}(\zeta_{K_{I}}(s))}\zeta_{K_{I}}(s)
\end{equation*}
 Recall
the following well known class number formula (see e.g.
\cite[Corollaire I.1.2]{Tate84})
\begin{equation*}
\zeta^{*}_{K_{I}}(0)=-\frac{h_{K_{I}}.\mathrm{Reg}_{K_{I}}}{|\mu(K_{I})|}
\end{equation*}
 This formula combined with the previous work gives
\begin{equation*}
(e_{S,r}U^{(r)}_{K}: \omega_{K}
e_{S,r}U^{(r)}_{K})=h_{K}\mathrm{Reg}_{K}\prod_{I\subset \{1,..,\mid
\mathrm{Ram}(K/k)\mid\}}h_{K_{I}}^{(-1)^{\mid I
\mid}}\mathrm{Reg}_{K_{I}}^{(-1)^{\mid I \mid}}.
\end{equation*}
Using Proposition \ref{Proposition index of Regulator} and Corollary
\ref{Corollary Sinnot}, we get
\begin{equation*}
(e_{S,r}\displaystyle{\bigcap^{r}_{G}}
U_{S,T}(K):e_{S,r}\mathrm{St}_{K,T})=h_{K}.(e_{S,r}\mathbb{Z}[G]
:e_{S,r}U^{(r)}_{K}).(e_{S,r}\displaystyle{\bigcap_{G}^{r}}U_{S,T}(K):
e_{S,r}\widetilde{\displaystyle{\bigwedge_{G}^{r}}}U_{S,T}(K)).
\beta_{K}.
\end{equation*}
where
\begin{equation}\label{beta}
\beta_{K}= c_{K}c_{K,r}^{-1}\prod_{I\subset \{1,..,\mid
\mathrm{Ram}(K/k)\mid\}}c_{K_{I}}^{(-1)^{|I|}}h_{K_{I}}^{(-1)^{\mid
I \mid}}.
\end{equation}
\end{proof}

\end{document}